\documentclass{amsart}
\usepackage{diagrams}
\usepackage{amsmath, amscd}
\usepackage{amssymb}
\usepackage{mathrsfs}
\usepackage[cmtip,all]{xy}
\usepackage{amsthm}

\newtheorem{theorem}[equation]{Theorem}

\newtheorem{prop}[equation]{Proposition}
\newtheorem{corol}[equation]{Corollary}

\theoremstyle{definition}
\newtheorem{definition}[equation]{Definition}

\theoremstyle{remark}
\newtheorem{remark}[equation]{Remark}
\numberwithin{equation}{section}

\newcommand{\CC}{\mathbb C}
\newcommand{\ZZ}{\mathbb Z}
\newcommand{\QQ}{\mathbb Q}
\newcommand{\oo}{\mathscr O}
\newcommand{\dd}{\mathscr D}
\newcommand{\scrt}{\mathscr T}
\newcommand{\mm}{\mathcal M}
\newcommand{\nn}{\mathcal N}
\newcommand{\VV}{\mathbb V}
\newcommand{\vv}{\mathcal V}
\newcommand{\fdot}{F_{\bullet}}
\newcommand{\LL}{ L}
\def\bysame{\leavevmode\hbox to3em{\hrulefill}\thinspace}
\newcommand{\abs}[1]{\lvert#1\rvert}

\begin{document}

\title{Vanishing and Injectivity theorems for Hodge Modules}
\author{Lei Wu}
\address{Deparment of Mathematics, Northwestern University, 2033 Sheridan Road, Evanston, IL 60208, USA }
\email{lwu@math.northwestern.edu}

\maketitle
\begin{abstract}

\setlength{\parindent}{0pt} \setlength{\parskip}{1.5ex plus 0.5ex
minus 0.2ex} 
We prove a surjectivity theorem for the Deligne canonical extension of a polarizable variation of Hodge structure with quasi-unipotent monodromy at infinity along the lines of Esnault-Viehweg. We deduce from it several injectivity theorems and vanishing theorems for pure Hodge modules. We also give an inductive proof of Kawamata-Viehweg vanishing for the lowest graded piece of the Hodge filtration of a pure Hodge module using mixed Hodge modules of nearby cycles.
\end{abstract}
\section{Introduction}
When $X$ is a smooth projective variety, the famous Kodaira-Nakano Vanishing theorem says that sufficiently high cohomologies vanish when $\Omega^p_X$ is twisted by any ample line bundle. Saito proved a more general vanishing theorem \cite{Sai90} using his theory of Hodge modules.
\begin{theorem}[\textbf{Kodaira-Saito Vanishing Theorem}]\label{KST}
Let $X$ be a complex projective variety with an ample line bundle $\LL$, and $M$ a mixed Hodge module on $X$. Then 
\[H^i(X, {\rm Gr}^{F}_{k}{\rm DR(M)}\otimes \LL)=0,    \ i>0\]
\end{theorem}
A detailed discussion of the proof of this theorem can be found in \cite{Pop}, another proof following the approach of Esnault-Viehweg in \cite{Sch14a}. If X is smooth, taking $M=\QQ_X^H:=(\omega_X, \fdot, \mathbb{Q}_X)$, the pure Hodge module corresponding to the trivial variation of Hodge structure on $X$, we have ${\rm Gr}^F_{-p}{\rm DR}(M)=\Omega^p_X[n-p]$, and so Saito's result implies Kodaira-Nakano vanishing. For arbitrary $M$, let $S(M)$ be the lowest graded piece of the Hodge filtration. In particular, \[H^i(X,S(M)\otimes \LL)=0, \  i>0,\] which specializes to Kodaira vanishing as well.\\
\indent On the other hand, Kodaira vanishing can be generalized by replacing ample divisors by nef and big divisors (or even $\mathbb{Q}$-divisors). This generalization is the so-called Kawamata-Viehweg vanishing theorem. In \cite{Pop}, M. Popa proved a version of Kawamata-Viehweg vanishing for some special Hodge modules, and he suggested that a better result would be true. In this paper I remove the extra hypothesis in \cite{Pop} and prove a Kawamata-Viehweg type statement for pure Hodge modules in full generality. 
\begin{theorem} \label{kvt1}
Let X be a complex projective variety, $\LL$ a nef and big line bundle, and $M$ a polarizable pure Hodge module with strict support $X$. Then 
\[H^i(X, S(M)\otimes \LL)=0,    \ i>\ 0\]
\end{theorem}
This can be approached in two different ways. I first provide an inductive approach, a strategy similar to Kawamata's original method (see also \cite[\S 4]{Laz}), based in this setting on an adjunction-type formula that involves the nearby cycle functor and mixed Hodge modules. This is presented in Section 5.\par
At the same time as this proof was completed, J. Suh proved a Nakano type vanishing for the Deligne canonical extension of a polarizable variation of Hodge structure in \cite{Suh}, which in particular implies the same Kawamata-Viehweg type result. (Note that Suh also proves other types of vanishing statements as well, that apply to all graded quotients in the Hodge filtration of the de Rham complex.) His idea is based on the Esnault-Viehweg \cite{EV} approach to vanishing theorems. Following J. Suh's proof and C. Schnell's Esnault-Viehweg type proof of Theorem \ref{KST} in \cite{Sch14a}, I extend this to a version of the injectivity theorem of Koll$\acute{\text a}$r and Esnault-Viehweg for the Deligne canonical extensions of certain polarizable variations of Hodge structures. This is presented as a surjectivity statement below and proved in Section 6.
\begin{theorem} \label{main}
Let $X$ be a smooth projective variety with a line bundle $\LL$, $D$ a reduced simple normal crossings divisor, and $\VV=(\vv, F^\bullet, \VV_\QQ)$ a variation of Hodge structure defined on $U=X\setminus D$. Assume  \[\LL^N=\mathcal{O}_X(D')\]
for some $N\gg0$ and an effective divisor $D'$ supported on $D$. If $E$ is an effective divisor supported on ${\rm Supp}(D')$. Then for all $i$, the natural map induced by $E$
\[H^i(X, {\rm Gr}^{\textup{first}}_F {\rm DR}_{(X, D)}(\tilde{\mathcal V})\otimes \LL^{-1}(-E))\longrightarrow H^i(X, {\rm Gr}^{\textup{first}}_F {\rm DR}_{(X, D)}(\tilde{\mathcal V})\otimes \LL^{-1})\] 
is surjective, where $\tilde\vv$ is the Deligne canonical extension of $\vv$, and ${\rm Gr}^{\textup{first}}_F {\rm DR}_{(X, D)}$ is the first non-zero graded piece of the logarithmic de Rham complex for $\tilde\vv$ (see Section 5).
\end{theorem}
A injectivity theorem for pure Hodge modules with strict support $X$ follows from it.
\begin{theorem} \label{injectivity}
Let X be a complex projective variety, $E$ an effective divisor, and $M$ a polarizable pure Hodge module with strict support $X$. If a line bundle $\LL$ is either nef and big or semi-ample and satisfying $\text{H}^0(X, \LL^m(-E))\neq0$ for some $m>0$,  then the natural map
\[H^i(X, S(M)\otimes \LL)\longrightarrow H^i(X, S(M)\otimes \LL(E))\]
is injective for all $i$.
\end{theorem}
Replacing $E$ by a sufficiently large multiple of an ample divisor, Theorem \ref{kvt1} follows by Serre vanishing, which provides the second proof. See also Corollary \ref{kv}. We also obtain a version of Kawamata-Viehweg vanishing for $\QQ$-divisors which contains the original version for $\omega_X$, as explained in Remark \ref{qdiv}.  \par
It is worth mentioning that, just like in Theorem \ref{KST}, the statement given here works on arbitrary (not necessary smooth) projective varieties $X$, due to the theory of Hodge modules on singular spaces; see Section 3.\par
In Sections 2, 3 and 4, we briefly review $\dd$-modules and Hodge modules. Some necessary theorems are presented for later use. In Section 5, I present the inductive proof of the Kawamata-Viehweg vanishing for pure Hodge modules using nearby cycles. Section 6 is dedicated to the proof of the injectivity theorem in the normal crossings case (Theorem \ref{main}). Section 7 deals with applications. 
\subsection*{Acknowledgements}   
I would like to thank my advisor, Mihnea Popa, for suggesting the problem and for many helpful discussions. I am also grateful to Christian Schnell for answering my questions. Finally, I thank Junecue Suh for correspondence and for sharing his preprint \cite{Suh}, which helped crucially with the second part of this work.
\section{$\dd$-modules and the de Rham complex}
In this section, I will recall some terminologies of $\dd$-modules and the Riemann-Hilbert correspondence between regular holonomic $\dd$-modules and perverse sheaves, which will be essentially used in the theory of Hodge modules. For a much detailed exposition of $\dd$-modules, we refer to \cite{HTT} and \cite{MS}. The former focuses more on algebraic $\dd$-modules, and the latter is mainly about the analytic story. 
\subsection{The side-change operator of $\dd$-modules}
Let $X$ be a complex manifold and let $\dd_X$ be the sheaf of differential operators. It is well-known that the category of left $\dd$-modules and the category of right $\dd$-modules are equivalent via the so called side-change operation,
\[\nn\longrightarrow\mm=\nn\otimes\omega_X, \] and its quasi-inverse
\[\mm\longrightarrow\nn=\mathcal{H}om_{\oo_X}(\omega_X, \mm),\]
where $\omega_X$ is the canonical sheaf.
When $\dd$-modules are filtered, the correspondence of filtrations under the equivalence is 
\[F_p(\mm)=F_{p+n}(\nn)\otimes\omega_X, \ \text{and}\  F_p(\nn)=F_{p-n}(\mm)\otimes\omega^{-1}_X,\] where $n=\text{dim}X$.
\subsection{The de Rham functor and the Riemann-Hilbert correspondence} The de Rham functor is defined to be
\[{\rm DR}(\nn):=[\nn\longrightarrow\Omega^1_X\otimes\nn\longrightarrow\cdot\cdot\cdot\longrightarrow\Omega^n_X\otimes\nn][n],\]
for left $\dd$-modules, where $\Omega^k_X$ is the sheaf of holomorphic $k$-forms on $X$. Also
\[{\rm DR}(\mm):=[\bigwedge^n\scrt_X\otimes\mm\longrightarrow\bigwedge^{n-1}\scrt_X\otimes\mm\longrightarrow\cdot\cdot\cdot\longrightarrow\mm][n],\]
for right $\dd$-modules, where $\scrt_X$ is the sheaf of vector fields on $X$. Both of the above complexes are concentrated in degree $-n,...,-1,0$. ${\rm DR}(\nn)$ (${\rm DR}(\mm$) respectively) is called the de Rham complex of $\nn$ ($\mm$ respectively). Clearly the de Rham functor is compatible with the side-change operation, i.e.
${\rm DR}(\nn)$ is canonically isomorphic to ${\rm DR}(\mm)$ provided $\mm=\nn\otimes \omega_X$, because of the canonical isomorphism
\[\bigwedge^{n-k}\scrt_X\otimes\omega_X\simeq\Omega^k.\] 
\par If the right $\dd$-module $\mm$ is filtered, then ${\rm DR}(\mm)$ is filtered naturally by 
\[F_p{\rm DR}(\mm):=[\bigwedge^n\scrt_X\otimes F_{p-n}\mm\longrightarrow\bigwedge^{n-1}\scrt_X\otimes F_{p-n+1}\mm\longrightarrow\cdot\cdot\cdot\longrightarrow F_p\mm][n].\]
Similarly, for filtered left $\dd$-module $\nn$,
\[F_p{\rm DR}(\nn):=[F_p\nn\longrightarrow\Omega^1_X\otimes F_{p+1}\nn\longrightarrow\cdot\cdot\cdot\longrightarrow\Omega^n_X\otimes F_{p+n}\nn][n].\]
Obviously, the filtered de Rham functor is also compatible with the side-change operation. Hence, it is not necessary to distinguish the right version and the left version of the (filtered) de Rham functor and (filtered) de Rham complexes.
The associated graded complexes for the filtration above are 
\[{\rm Gr}^F_p{\rm DR}(\mm):=[\bigwedge^n\scrt_X\otimes {\rm Gr}^F_{p-n}\mm\longrightarrow\bigwedge^{n-1}\scrt_X\otimes {\rm Gr}^F_{p-n+1}\mm\longrightarrow\cdot\cdot\cdot\longrightarrow {\rm Gr}^F_p\mm][n],\]which are complexes of $\oo_X$-modules.

\par The Riemann-Hilbert correspondence says that the de Rham functor induces an equivalence between the category of regular holonomic $\dd$-modules and the category of perverse sheaves. In particular, holomorphic vector bundles with flat connections correspond to local systems under this equivalence. This is just part of the general Riemann-Hilbert correspondence. See for instance \cite[\S 7 ]{HTT} for the whole statement. \par
Convention: We denote increasing filtrations (decreasing filtrations respectively) by $\fdot$ ($F^{\bullet}$ respectively), and the associated graded objects by ${\rm Gr}^F$(${\rm Gr}_F$ respectively).
\section{Hodge modules}
This section will be devoted to briefly recall Morihiko Saito's theory of Hodge modules. I will only mention basic information of Hodge modules and important theorems that will be needed later on. The main two references are Saito's original paper \cite{Sai88} and \cite{Sai90}. Another useful reference is the recent survey \cite{Sch14}. Since Hodge modules are originally defined for complex manifolds, when we say Hodge modules on some complex algebraic variety, it means on the underlying analytic space. Since we are about to work on projective varieties, all the ingredients of Hodge modules will be algebraic by the GAGA principle. Varieties always mean reduced and irreducible complex algebraic schemes or analytic schemes in this paper.
\subsection{Pure Hodge modules}
Let $X$ be a smooth complex variety or complex manifold.  A variation of Hodge structure $\mathbb V$ of weight $l$ is
\[\mathbb V=(\mathcal V, F^{\bullet}, \mathbb V_{\QQ}),\]
consisting of 
\begin{itemize}
\item a $\QQ$-local system $\mathbb V_{\QQ}$;
\item a finite decreasing filtration $F^{\bullet}$ of the vector bundle $\mathcal V=\mathbb V_{\QQ}\otimes\oo_X$ by subbundles,
satisfying 
\item $\forall x\in X$, $\mathbb{V}_x=(\mathcal{V}_x, F^{\bullet}_x, \mathbb{V}_{\QQ, x})$ is a  Hodge structure of weight $l$;
\item the filtration $F^{\bullet}$ satisfies the Griffiths transversality condition, namely for the induced connection $\nabla$,
\[\nabla(F^p)\subset \Omega^1\otimes F^{p-1}.\]
\end{itemize}
Additionally, a polarization of a variation Hodge structure $\mathbb V$ of weight $l$ is a morphism 
\[Q: \mathbb V_{\QQ}\otimes \mathbb V_{\QQ}\longrightarrow\QQ(-l),\] 
such that $Q$ induces a polarization of $\mathbb V_x$ for every $x\in X$, where $\QQ(-l)=(2\pi i)^{-l}\QQ$\\\\
Note: If we let $F_p\mathcal V= F^{-p}\mathcal V$, then the last requirement means exactly that the new filtration $F_{\bullet}$ is good for left $\dd$-module $\mathcal V$. From now on, $\fdot\mathcal V$ of a variation of Hodge structure always means this induced increasing filtration.\par
Saito's theory generalizes variations of Hodge structure by allowing filtered holonomic $\dd$-modules with $\QQ$-structure instead.
Here a filtered holonomic $\dd$-modules with $\QQ$-structure is a triple 
\[M=(\mm, \fdot\mm, K)\]
where $(\mm, \fdot\mm)$ is a holonomic right $\dd$-module with a good filtration, and $K$ is a $\QQ$-perverse sheaf, satisfying 
\[{\rm DR}(\mm)\simeq \CC\otimes_{\QQ} K.\]
We use right $\dd$-modules because it is more natural to define the direct image funtor for right $\dd$-modules. See Section \ref{directimage}. From now on, all $\dd$-modules will mean right $\dd$-modules unless stated explicitly. I will also use the expression "a filtered left $\dd$-module underlying a Hodge module (pure or mixed)" if it is so after side-change. \par
Saito constructed an abelian category $\text{HM}(X,l)^p$ of polarizable pure Hodge modules of weight $l$ for smooth complex variety $X$ (or complex manifold) in \cite{Sai88} which is a fully faithful subcategory of category of filtered holonomic $\dd$-modules with $\QQ$-structures. (M. Saito constructed pure Hodge modules first, but we are only interested in polarizable ones.) To be more precise, $\text{HM}(X,l)^p$ is semi-simple, i.e.
\[\text{HM}(X, l)^p=\bigoplus_{Z\subseteq X} \text{HM}^Z(X, l)^p,\]
direct sum over all irreducible subvariety of X. Here $\text{HM}^Z(X, l)^p$ is of objects extended from polarizable variations of Hodge structures of weight $l-\text{dim}Z$ on Zariski open subset of the smooth locus of $Z$. From this, it is clear that when $X$ is a point, $\textup{HM}(X, l)^p$ is just the category of polarizable Hodge structures of weight $l$. \par
If $Y$ is an irreducible subvariety of $X$, Hodge modules on $Y$ can be defined as,
\begin{definition}[Pure Hodge Modules on Singular Spaces] \label{singhm}
\[\textup{HM}(Y, l)^p:=\textup{HM}_Y(X, l)^p,\]
where the right-hand side is the category of polarizable pure Hodge modules of weight $l$ on $X$ supported on $Y$. 
\end{definition}
In this way, the underlying $\dd$-module of $M\in\textup{HM}(Y, l)^p$ makes sense. So do ${\rm DR}(\mm)$, $\fdot {\rm DR}(M)$ and ${\rm Gr}^F_{\bullet}{\rm DR}(\mm)$. In fact, by the filtered version of Kashiwara's equivalence, $\textup{HM}(Y, l)^p$ doesn't depend on embeddings of $Y$ into ambient space. Moreover, ${\rm Gr}^F_k{\rm DR}(\mm)$ are well-defined complexes of coherent sheaves on $Y$, independent of the choice of an embedding. Namely, $\textup{HM}(Y, l)^p$ and ${\rm Gr}^F_k{\rm DR}(\mm)$ are intrinsically defined for $Y$. See \cite[\S14]{Sch14} for details. Indeed, only local embeddings are needed. Then gluing local data together 
gives rise to the definition of $\textup{HM}(Y, l)^p$ \cite{Sai88}. 
\par 
A polarizable pure Hodge module $M$ of weight $l$ is of strict support $Z$, or strictly supported on $Z$, if $M\in \text{HM}^Z(X, l)^p$. Indeed, it is called strict because its underlying perverse sheaf is the intersection complex of the local system. Saito also proved the following equivalence, which will play an important role later on.
\begin{theorem}[{\cite[Theorem 3.21]{Sai90}}] \label{geneq}
Let $X$ be a complex variety of dimension $n$. Then the following two categories are equivalent:
\[\textup{HM}^X(X, l)^p\simeq\textup{VHS} ^p_{gen}(X, l-n),\]
where $\textup{VHS}^p_{gen}(X, l-n)$ is the inductive limit of $\textup{VHS}^p(U, l-n)$ the categories of polarizable variations of Hodge structures of weight $l-n$ on smooth dense Zariski open subsets $U$.  
\end{theorem}

\subsection{Mixed Hodge modules}
In \cite{Sai90}, Saito constructed another abelian category, the category of mixed Hodge modules. Similar to mixed Hodge structures generalizing pure Hodge structures, the basic input of a mixed Hodge module on a smooth variety $X$ is a filtered holonomic $\dd$-module with $\QQ$-structure, $(\mm, \fdot\mm, K)$ plus a finite increasing filtration $W_{\bullet}$ on $K$, which in turn induces a finite filtration $W_{\bullet}\mm$ for $\mm$ by the Riemann-Hilbert correspondence. This $W$-filtration is called the weight filtration, while the $F$-filtration is called the Hodge filtration.
\par Saito first constructed the category of weakly mixed Hodge modules (graded polarizable), $\text{MHW}(X)^p$. An $M=(\mm, \fdot\mm, K, W_{\bullet})$ is a weakly mixed Hodge module (graded polarizable) if \[ ({\rm Gr}^W_l\mm, \fdot, {\rm Gr}^W_lK)\in\text{HM}(X, l)^p,\] 
where $\fdot$ is the induced filtration. 
If in addition, such an $M$ also satisfies some "admissibility" condition (originally defined in \cite{SZ85} for admissible variations of mixed Hodge structures), then $M$ is a mixed Hodge module. See \cite[\S20]{Sch14} or originally \cite[\S2]{Sai90} for a detailed discussion of the definition of mixed Hodge modules. \par
Denote the category of graded polarizable mixed Hodge modules by $\text{MHM}(X)^p$. Mixed Hodge modules have many good properties, one of which is that $\text{MHM}(X)^p$ is stable by $j_*j^{-1}$ (see \cite[Proposition 2.11]{Sai90}), where $j$ is the open embedding
\[j:X\setminus D\longrightarrow X\] 
for some effective divisor $D$. This is called localizations of mixed Hodge modules in \cite{Pop}. At the level of perverse sheaves, this functor is precisely 
\[Rj_*j^{-1}(K).\]
If $M\in \text{MHM}(X)^p$, it is hard to compute $j_*j^{-1}M$ explicitly in general, even for $j_*j^{-1}\QQ_X^H[n]$. But in some special situation like for instance $D$ is a smooth irreducible divisor, $j_*j^{-1}\QQ_X^H[n]$ is understood very well (\cite[\S 5]{Pop}). I will calculate some special examples of mixed Hodge modules of this type in the next section.   
\subsection{Direct image functor and the Stability Theorem}\label{directimage}
Let $f: X\longrightarrow Y$ be a projective morphism of smooth complex varieties (or complex manifolds), and let $M$ be a Hodge module (pure or mixed) on $X$. The direct image of $M$ is \[f_*M:=(f_+(\mm, \fdot\mm), Rf_*K),\]
i.e. combining the direct image of the filtered $\dd$-module and the direct image of the perverse sheaf.
Detailed discussion about direct image functors can be found in \cite[\S26, \S28]{Sch14}. One of the main results of M. Saito's theory is the following theorem about the direct image functor.
\begin{theorem}[Stability Theorem and Decomposition Theorem \cite{Sai88}] \label{stability}
Let $f:X\longrightarrow Y$ is a projective morphism between complex manifolds, and let \[M=(\mm, \fdot\mm, K)\in\textup{HM}^Z(X, l)^p,\] for some subvariety Z of $X$. Then,\\
(1) $f_+(\mm, \fdot\mm)$ is strict, and $\mathcal H^if_*M\in\textup{HM}(Y, n+i)^p$;\\
(2) $f_*M\simeq \bigoplus\mathcal H^if_*M[-i]$ (non-canonical). 
\end{theorem}
Statement (1) is called the Stability Theorem. The statement in \cite{Sai88} is more general, but this is all I need in this article. The second statement is called the Decomposition Theorem, and is easily deduced from the first because $f_*M$ is a complex of pure weights (see \cite[$\S$14.1.2]{PS}). This is also the reason why the quasi-isomorphism in (2) isn't canonical (\cite[Corollary 14.4]{PS}).
\begin{remark}
When $(\mm, \fdot\mm)$ underlies a graded-polarizable weakly mixed Hodge module $M$, Stability Theorem is still true (see \cite[Theorem 2.14]{Sai90}). But since $f_*M$ may not be of pure weights, Decomposition Theorem doesn't hold for weak mixed Hodge modules or even mixed Hodge modules. 
\end{remark}
As an easy consequence of the Stability Theorem, one has the following useful corollary:
\begin{corol}[{\cite[Example 27.1]{Sch14}}] \label{degen}
Let $f: X\longrightarrow \bullet$ be the constant morphism from a smooth projective variety, with $(\mm, \fdot\mm)$ a filtered holonomic $\dd$-module underlying some graded-polarizable weakly mixed Hodge module. Then the Hodge-de Rham spectral sequence
\[E^{p, q}_1=H^{p+q}(X, {\rm Gr}^F_{-p}{\rm DR}(\mm))\implies H^{p+q}(X, {\rm DR}(\mm))\]
degenerates at $E_1$.  
\end{corol} 
Since the filtration is bounded below, it is useful to emphasize on where the filtration starts.
\begin{definition}
Suppose $M=(\mm, \fdot\mm, K)$ is a pure Hodge module on a smooth complex variety $X$. Set
\[p(M):=\text{min}\{p\ |\ F_{p}\mm\neq 0\},\]
and
\[S(M):=F_{p(M)}\mm= {\rm Gr}^F_{p(M)}{\rm DR}(\mm).\]
\end{definition}
$S(M)$ also makes sense when $X$ is singular, because $S(M)= {\rm Gr}^F_{p(M)}{\rm DR}(\mm)$ and the latter one is intrinsic for $M$ as explained in Definition \ref{singhm}. $S(M)$ will play a similar role as the canonical sheaf does in vanishing theorems.
From the Decomposition Theorem, Saito also proved,
\begin{prop}[{\cite[Theorem 28.1]{Sch14}}] \label{kol}
Under the setting of Theorem \ref{stability}, 
\[Rf_*{\rm Gr}^F_p{\rm DR}(\mm)\simeq\bigoplus_i{\rm Gr}^F_p{\rm DR}(\mm_i)[-i].\] In particular,
\[Rf_*S(M)\simeq \bigoplus_iF_{p(M)}\mm_i[-i],\]
where $(\mm_i, \fdot\mm_i):=\mathcal H^if_+(\mm, \fdot\mm)$.
\end{prop}
\subsection{$V$-filtrations and the torsion-freeness of $S(M)$}
The $V$-filtration is fundamentally important in Saito's theory of Hodge modules. It is used in this section to prove the torsion-freeness of $S(M)$ that due to Saito because of completeness.  \par
Let $X$ be a smooth complex variety or complex manifold of dimension $n$ and let $X_0$ be a smooth divisor of $X$ with local defining equation $t$. $\dd_X$ is filtered by
\[V_{i}\dd_X=\{P\in\dd_X| P\cdot\mathcal{I}^j_{X_0}\subseteq\mathcal{I}^{j-i}_{X_0}\},\]
where $\mathcal{I}_{X_0}$ is the ideal sheaf of $X_0$ with the convention that $\mathcal{I}^j_{X_0}=\oo_X$ for $j\leq 0$.
\begin{definition}[$V$-filtration]
A $V$-filtration of a coherent right $\dd$-module $\mm$ along $X_0$ is an increasing filtration $V_{\alpha}\mm$ indexed by rational number $\alpha$ satisfying, 
\begin{itemize}
\item The filtration is exhaustive, and each $V_{\alpha}\mm$ is a coherent $V_0\dd_X$-module. \\
\item $V_{\alpha}\mm\cdot V_i\dd_X\subseteq V_{\alpha+i}\mm$ for every $\alpha\in\QQ$ and $i\in\ZZ$; and 
\[V_{\alpha}\mm\cdot t=V_{\alpha-1}\mm, \  \text{for}\  \alpha<0\]
\item The action of $t\partial_t-\alpha$ on $Gr^V_{\alpha}\mm$ is nilpotent for any $\alpha$, where $Gr^V_{\alpha}\mm:=V_{\alpha}\mm/V_{<\alpha}\mm$ and $V_{<\alpha}\mm=\cup_{\beta<\alpha}V_\alpha\mm$.
\end{itemize}
\end{definition}
In fact, the morphism of the action $t$ in the second requirement of the above definition is an isomorphism (\cite[Lemme 3.1.4]{Sai88}), i.e.
\[t: V_{\alpha}\mm\xrightarrow{\simeq} V_{\alpha-1}\mm\]
is an isomorphism for all $\alpha<0$. \par
The $V$-filtration is a refinement of the Kashiwara-Malgrange filtration. It is known that the Kashiwara-Malgrange filtration exists and is unique for any regular holonomic $\dd$-module. It can be refined to a $V$-filtration if the monodromy action is quasi-unipotent. Hence in particular, the $V$-filtration exists for any Hodge module. See \cite[\S8]{Sch14} for a detailed exposition.  \par
For filtered holonomic $\dd$-modules, one needs to consider the compatibility of the $V$-filtration and the $\fdot$ filtration. Such compatibility condition is called "quasi-unipotent and regular along $X_0$ or some locally defined holomorphic function $f$" (see \cite[\S3.2]{Sai88}). Pure Hodge modules are quasi-unipotent and regular along any locally defined holomorphic function $f$. This condition makes the Hodge filtration of any Hodge module only depends on its restriction to the open subset complementary to the divisor defined by $f$. To be precise, if $M=(\mm, \fdot\mm, K)$ is a pure Hodge module of strict support $Z$ on $X$, then for holomorphic function $f:X\longrightarrow \CC$ whose restriction to $Z$ is not constant, 
\begin{equation} \label{filtr}
F_{p}\mm_f=\sum_{i=0}^\infty (V_{<0}\mm_f\cap j_*j^*F_{p-i}\mm_f)\partial_t^i,
\end{equation}
where $(\mm_f, \fdot\mm_f)=i_{f,+}(\mm, \fdot\mm)$, $i_f$ is the graph embedding of $f$, and $t$ is the coordinate of $\CC$. 
\begin{prop}\label{otorsionfree}
Let $X$ be a complex variety (irreducible), and $M$ a pure Hodge module strictly supported on $X$. Then $S(M)$ is torsion-free.
\end{prop}
\begin{proof}
Since torsion-freeness is local, we can assume $X\hookrightarrow Y$ embedded into a complex manifold Y, and the underlying filtered $\dd$-module is $(\mm, \fdot\mm)$ on $Y$. We know 
\[S(M)=F_{p(M)}\mm\]
is a coherent $\oo_X$-module. Suppose $\mathcal T$ is the torsion submodule of $S(M)$, and $\mathcal T$ is annihilated by $\tilde f$ which can be lifted to a holomorphic function $f:Y\longrightarrow \CC$ on Y. By equation \eqref{filtr},
\[S(M)=F_{p(M)}\mm_f=V_{<0}\mm_f\cap j_*j^*F_{p(M)}\mm_f.\]
And the action of $\tilde f$ on $S(M)$ is exactly the action of $t$ on $F_{p(M)}\mm_f$. But since  $t$ acts on $V_\alpha$ injectively for $\alpha<0$, $\mathcal T$ must be 0.
\end{proof}
Because of the compatibility of the $V$-filtration and the Hodge filtration, Saito also proved the following theorem about $p(M)$.
\begin{prop}[{\cite[Proposition 2.6]{Sai91}}] \label{p(M)}
Suppose $f:X\longrightarrow Y$ is a projective morphism of complex manifolds with $M\in{HM^Z(X, l)^p}$ for some subvariety $Z$. If $M'$ is a direct summand of $\mathcal H^jf_*M$ and its strict support  $Z'\neq f(Z)$, then 
\[p(M')> p(M).\]
\end{prop}
Combining Proposition \ref{otorsionfree} and Proposition \ref{p(M)}, one obtains,
\begin{corol}\label{torsionfree}
Let $f:X\longrightarrow Y$ be a surjective projective morphism of complex varieties, and let $M$ be a polarizable pure Hodge module strictly supported on $X$. Then
$R^if_*S(M)$ is torsion-free for all $i$.
\end{corol}
This corollary was first conjectured by Koll\'{a}r as generalizations of his theorem about higher direct images of dualizing sheaves \cite{Kol}, and was proved by M. Saito in \cite{Sai91}. 
\begin{corol}\label{hdim}
Assume $f:X\longrightarrow Y$ is a birational morphism of complex projective varieties, and $M\in HM^X(X, l)^p$. Then
\[R^if_*S(M) = \left\{ \begin{array}{ll}
         S(M') & \mbox{ $i = 0$}\\
        0 & \mbox{ $i<0$}\end{array} \right., \]
where $M'$ is the direct summand of $\mathcal H^0f_*M$ strictly supported on Y.
\end{corol}
\begin{proof}
The fact that $R^0f_*S(M)=S(M')$ follows from Proposition \ref{p(M)}. Now $R^if_*S(M)=0$ for $i>0$ because $R^if_*S(M)$ is both torsion and torsion-free. 
\end{proof}
\subsection{Non-characteristic inverse image} \label{nonchar}
Let $f:X\longrightarrow Y$ be a morphism of complex manifolds. We have the following two morphisms of cotangent bundles,
\[p_2: X\times_Y T^*Y\longrightarrow T^*Y\] and
\[df^*: X\times_Y T^*Y\longrightarrow T^*X\] given by differential, $df^*(x, \omega)=df^*_x(\omega)$.

\begin{definition}[Non-Characteristic morphism {\cite[\S3.5]{Sai88}}]
Let $(\mm, \fdot\mm)$ be a filtered coherent $\dd_Y$-module. $f:X\longrightarrow Y$ is said to be non-characteristic for $(\mm, \fdot\mm)$ if the following two conditions are satisfied:
\begin{itemize}
\item $L^if^*{\rm Gr}^F_p(\mm)=0$ for all $i\neq0$ and all $k$;\\
\item $df^*|_{p_2^{-1}\text{Ch}(\mm)}$ is finite.
\end{itemize}
If $f$ is a closed embedding, one says that $X$ is non-characteristic for $(\mm, \fdot\mm)$ if $f$ is so.\par
The first condition means that the pull-back of the filtration is still a filtration of the pull-back of $\mm$;
the second says the pull-back filtration is also good. Hence one gets the filtered pull-back $f^*(\mm, \fdot\mm)=(\tilde\mm, \fdot\tilde\mm)$ when $f$ is non-characteristic for $(\mm, \fdot\mm)$ by
\[\tilde\mm=f^{-1}\mm\otimes_{f^{-1}\oo_Y}\omega_{Y/X}\]
and
\[F_p\tilde\mm=f^{-1}F_{p+d}\mm\otimes_{f^{-1}\oo_Y}\omega_{Y/X},\]
where $d=\text{dim}X-\text{dim}Y$.
Here the degree of the filtration of $\tilde\mm$ is shifted because $\mm$ is a right $\dd$-module. $\tilde\mm$ is holonomic if $\mm$ is so (\cite[Lemme 3.5.5]{Sai88}). If $(\mm, \fdot\mm)$ underlies a polarizable pure Hodge module $M$ of weight $l$ on $Y$, then $f^*(\mm, \fdot\mm)$ underlies a polarizable pure Hodge module of weight $l+d$, denoted by $f^*M$ (\cite[Theorem 9.3]{Sch14a}). 
\end{definition}
\section{Hodge modules and the Deligne extension}\label{snc}
Let $X$ be a complex smooth projective variety, and let $D=\sum^r_{i=1} D_i$ be a reduced simple normal crossings divisor with irreducible components $D_i$. Set $U=X\setminus D$, and 
\[j:U\hookrightarrow X,\] the open embedding. \par 
Assume that $\mathbb V=(\mathcal V, F^{\bullet}\mathcal V, \mathbb{V}_\QQ)$ is a polarizable variation of Hodge structure on $U$. Then it is well-known that $\mathcal V$ extends to vector bundles with flat connections with logarithmic poles along $D$. Such extension is unique if the eigenvalues of the residue along each $D_i$ are required to be in a fixed strip of $\CC$ of length $1$ (\cite[Theorem 5.2.17]{HTT}). Because of the monodromy theorem (see \cite{Sch}), the eigenvalues are in fact rational numbers in this setting. The Deligne canonical extension is the extension with eigenvalues in $[0, 1)$, denoted by $\mathcal{\tilde V}$. $\mathcal{\tilde V}$ is filtered by 
\begin{equation}\label{filtrDe}
F_p\mathcal{\tilde V}:=\mathcal{\tilde V}\cap j_*(F_p\mathcal V),
\end{equation}
where $F_p\mathcal V=F^{-p}\mathcal V$. Because of Schmid's theorem on nilpotent orbits, filtrations of all Deligne extensions are locally free (see \cite[\S3.b]{Sai90}).
Hence we have the filtered logarithmic de Rham complex for $\tilde{\mathcal V}$, similar to the filtered de Rham complex,
\[{\rm DR}_{(X,D)}(\tilde{\mathcal V}, \fdot):=\lbrack\tilde{\mathcal V}\longrightarrow \Omega^1({\rm log}~D)\otimes\tilde{\mathcal V}\longrightarrow\cdots\longrightarrow\Omega^n({\rm log}~D)\otimes\tilde{\mathcal V}\rbrack [n]. \] with filtration
\[F_p {\rm DR}_{(X,D)}(\tilde{\mathcal V}, \fdot):=\lbrack F_p\tilde{\mathcal V}\longrightarrow \Omega^1({\rm log}~D)\otimes F_{p+1}\tilde{\mathcal V}\longrightarrow\cdots\longrightarrow\Omega^n({\rm log}~D)\otimes F_{p+n}\tilde{\mathcal V}\rbrack [n], \]
where $n=\text{dim}X$. 
\par
Then $\tilde{\mathcal V}(*D)=\tilde{\mathcal V}\otimes\oo_X(*D)$ is the regular meromorphic connection extending $\mathcal V$, which is a left regular holonomic $\dd$-module a priori. According to \cite[S3.b]{Sai90}, $\tilde{\mathcal V}(*D)$ can be filtered by 
\[F_p\tilde{\mathcal V}(*D)=\sum F_{p-i}\dd_X\cdot F_i\tilde{\mathcal V}_1,\]
where $\tilde{\mathcal V}_1$ is another extension of $\mathcal V$ with eigenvalues in $[-1, 0)$. \par
By Theorem \ref{geneq}, one denotes the pure Hodge module of strict support X corresponding to $\mathbb V$ by $M$. Since $\mathbb V$ is defined outside of a normal crossings divisor, $M$ can be described explicitly in terms of a Deligne extension  (see \cite[Theorem 3.20]{Sai90}). In particular,
\[S(M)=\omega_X\otimes F_{p(M)+n}\mathcal V_2,\]
where $\mathcal V_2$ is another Deligne extension with eigenvalues in $(-1, 0]$. Hence $S(M)$ is locally free. Saito proved that the filtered left $\dd$-module $(\tilde{\mathcal V}(*D), \fdot\tilde{\mathcal V}(*D))$ underlies $j_*j^{-1}M$. This point is clear at least for $\tilde{\mathcal V}(*D)$ without filtration. It is the underlying left $\dd$-module of $j_*j^{-1}M$ because ${\rm DR}(\tilde{\mathcal V}(*D))\simeq Rj_*\mathbb V_\CC$ (see \cite[Theorem 5.2.24]{HTT}). 
Furthermore, by \cite[Proposition 3.11 (3.11.4)]{Sai90}, there is a filtered quasi-isomorphism 
\begin{equation} \label{qi}
{\rm DR}_{(X,D)}(\tilde{\mathcal V}, \fdot\tilde{\mathcal V})\simeq {\rm DR}(\tilde{\mathcal V}(*D), \fdot\tilde{\mathcal V}(*D)).
\end{equation} This filtered quasi-isomorphism will be used repeatedly. \par
Since $\VV$ is finitely filtered by $\fdot\VV$, it is also useful to define (following \cite{Sai91})
\[q(M)=q(\VV):=\text{max}\{p\  |\ {\rm Gr}_p^F(\VV)\neq 0\}.\]
For any pure Hodge module N strictly supported on $X$, besides $S(N)$ we define (following \cite{Sai91})
\begin{equation} \label{Q(M)}
Q_X(N):=\mathbb{D}(S_X(N^*))=R\mathcal{H}om_{\oo_X}(S_X(N^*), \omega_X[n]),
\end{equation}
where $N^*$ is the extension of the dual of the variation of Hodge structure defined on a Zariski open dense subset in the sense of Theorem \ref{geneq}, and $\mathbb{D}$ is the Grothendieck duality functor.  When $X$ is a singular projective variety, $Q_X(N)$ is defined to be 
\[Q_X(N)=R\mathcal{H}om_{\oo_Y}(S_X(N^*), \omega_Y[\text{dim}Y])\] for some smooth $Y$ containing $X$ as a subvariety. This definition is independent with $Y$ because of the well-known fact that the pushforward functor and the Grothendieck duality functor commute.   \par
By \cite[Lemma 2.4]{sai92}, under this normal crossings assumption,
\[Q_X(M)\simeq\frac{\tilde{\mathcal V}}{F_{q(M)-1}\tilde {\mathcal V}}[n].\] Hence one obtains, 
\begin{equation} \label{dualiso}
Q_X(M)\simeq\footnote{ I learned this point from J. Suh.}\ {\rm Gr}^F_{q(M)}{\rm DR}_{(X,D)}(\tilde{\mathcal V}, \fdot)\simeq {\rm Gr}^F_{q(M)}{\rm DR}(\tilde{\mathcal V}(*D), \fdot).
\end{equation}
 The second one is the quasi-isomorphism induced by the filtered quasi-isomorphism \eqref{qi}.
\section{An inductive proof of Kawamata-Viehweg type vanishing theorem}
In this section, I prove the Kawamata-Viehweg type vanishing theorem inductively by the method similar to Kawamata's original approach (see \cite[the proof of Theorem 4.3.1]{Laz}). It will be proved alternatively by the injectivity theorem in Section 7. See the first half of Corollary \ref{kv}.
\begin{theorem}\label{KVV}
Let $X$ be a complex projective variety with polarizable pure Hodge module $M$ strictly supported on $X$ and let $\LL$ be a nef and big line bundle on $X$. Then 
\[H^i(X, S(M)\otimes\LL)=0, \ \ i>0.\]  
\end{theorem}
\begin{proof}
The proof will be divided into two steps. Step 1 is a Norimatsu-type statement (see \cite[Lemma 4.3.5] {Laz}); Step 2 is to reduce the general statement to the case of Step 1.  \\\\
$Step$ 1. In this step, we make the following assumption:
\begin{itemize}
\item X is smooth;\\
\item M is extended from a polarizable variation of Hodge structure $\mathbb V=(\mathcal V, F^{\bullet}\mathcal V, \mathbb{V}_\QQ)$ on $U$, with $D=X\setminus U$ a normal crossings divisor; \\
\item The monodromy of $\VV$ along each irreducible component of $D$  is unipotent;\\
\item $\LL\simeq\oo_X(A+E)$ with $A$ an ample divisor and $E=\sum_{i=1}^tE_i$ a reduced simple normal crossings divisor, and ${\rm Supp}(E)\subset D$. 
\end{itemize}
\par
Denote the Deligne canonical extension of $\vv$ by $\tilde{\vv}$, as in section \ref{snc}. Then the third assumption means the residue of $\tilde\vv$ along each irreducible component of $D$ is nilpotent. As explained in Section \ref{snc},
\[Q(M)=\dfrac{\tilde\vv}{F_{q(M)-1}\tilde{\vv}}[\text{dim}X]={\rm Gr}^F_{q(M)}\tilde\vv[\text{dim}X]. \]
Set $Q'=Q(M)[-\text{dim}X];$ it is a locally free sheaf. \par
Under these assumptions, we show that
\[H^i(X, Q'(-A-E))=0, \ i<\text{dim}X. \]
We use induction on $t$, the number of components of $E$, and on dimension. The case $t=0$ is Kodaira-Saito vanishing because of the filtered quasi-isomorphism \eqref{qi}. Assuming the result known for $E$ with $\leq k$ components, consider the short exact sequence,
\[0\longrightarrow \oo_X(-A-\sum_{i=1}^{k+1}E_i)\longrightarrow \oo_X(-A-\sum_{i=1}^kE_i)\longrightarrow \oo_{E_{k+1}}(-A-\sum_{i=1}^kE_i)\longrightarrow 0.  \] Tensoring with $Q'$ gives,
\begin{equation}\label{ses}
0\longrightarrow Q'(-A-\sum_{i=1}^{k+1}E_i)\longrightarrow Q'(-A-\sum_{i=1}^kE_i)\longrightarrow Q'|_{E_{k+1}}\otimes\oo_{E_{k+1}}(-A-\sum_{i=1}^kE_i)\longrightarrow 0. 
\end{equation}
By the inductive assumption,
\begin{equation}\label{ind1}
H^i(X, Q'(-A-\sum_{i=1}^kE_i))=0,\ i<\text{dim}X.
\end{equation}
Now $\tilde{\vv}|_{E_{k+1}}$ is not a Deligne canonical extension of some variation of Hodge structure defined on some open set of $E_{k+1}$. However, the residue along $E_{k+1}$ gives rise to a finite increasing filtration $W_{\bullet}$ for $\tilde{\vv}|_{E_{k+1}}$, the monodromy weight filtration. Moreover, $Gr^{W}_l(\tilde{\vv}|_{E_{k+1}})$ for any $l$, is the Deligne canonical extension of some polarizable variation of Hodge structure defined on $E_{k+1}\setminus D$, following from the multi-valued SL$_2$-orbit theorem \cite{CKS}. The filtration on $Gr^{W}_l(\tilde{\vv}|_{E_{k+1}})$ induced from the variation of Hodge structure is precisely the filtration induced from $\tilde\vv$, i.e.
\[F_p{\rm Gr}^{W}_l(\tilde{\vv}|_{E_{k+1}})=\dfrac{F_p\tilde\vv|_{E_{k+1}}\cap W_l}{F_p\tilde\vv|_{E_{k+1}}\cap W_{l-1}}.\] 
Namely $\tilde{\vv}|_{E_{k+1}}$ is the Deligne canonical extension of an admissible variation of mixed Hodge structure because of the unipotency assumption of the monodromy. See \cite[Theorem 3.20]{Bru},\footnote{This process is called the graded nearby-cycle functor in \cite{Bru}.} or \cite[Proposition 4.3]{FFS}.\par
By the inductive assumption again,
\[H^i(E_{k+1}, {\rm Gr}^F_{q(M)}{\rm Gr}^W_l\tilde{\vv}|_{E_{k+1}}\otimes\oo_{E_{k+1}}(-A-\sum_{i=1}^kE_i))=0,\] for $i<\text{dim}X-1$ and all $l$. By a simple calculation, there are short exact sequences for any $p$ and $l$,
\[0\longrightarrow\dfrac{F_p\tilde\vv|_{E_{k+1}}\cap W_{l-1}}{F_{p-1}\tilde\vv|_{E_{k+1}}\cap W_{l-1}}\longrightarrow\dfrac{F_p\tilde\vv|_{E_{k+1}}\cap W_l}{F_{p-1}\tilde\vv|_{E_{k+1}}\cap W_{l}}\longrightarrow {\rm Gr}^F_{p}{\rm Gr}^W_l\tilde{\vv}|_{E_{k+1}}\longrightarrow 0.\] Using the long exact sequences of cohomology associated to the above short exact sequences, one sees that
\begin{equation}\label{ind2}
H^i(E_{k+1}, Q'|_{E_{k+1}}\otimes\oo_{E_{k+1}}(-A-\sum_{i=1}^kE_i))=0, \ i<\text{dim}X-1 .
\end{equation}
Using the long exact sequence of cohomology associated to the short exact sequence \eqref{ses}, together with \eqref{ind1} and \eqref{ind2}, we get
\[H^i(X, Q'(-A-\sum_{i=1}^{k+1}E_i))=0, \ i<\text{dim}X. \]
$Step$ 2. Since the Leray spectral sequence degenerates at $E_1$ because of Corollary \ref{hdim}, by passing to a log-resolution (see also the proof of Theorem \ref{semiample}), it suffices to assume that $X$ is smooth, $M$ is extended from a polarizable variation of Hodge structure $\VV$ defined on $U$ with $D=X\setminus U=\sum D_i$ a reduced simple normal crossings divisor, and $\LL^m=\oo(A+E)$ with $A$ an ample divisor and $E$ an effective divisor with support contained in $D$. Denote 
\[E=\sum^t_{i=1} \alpha_iD_i,\ \alpha=\alpha_1\cdot...\cdot \alpha_t, \ \text{and}\  \alpha_i'=\alpha/\alpha_i,\] with $\alpha_i>0$.
Using a Kawamata covering \cite{Kaw},  we can construct a finite flat covering 
\[f:Y\longrightarrow X,\]
such that $f^*D$ still has simple normal crossings support, 
\[f^*D_i=m\alpha_i'D_i',\] for $i=1,... ,t$
and $E'=\sum^t_{i=1}D'_i$ is a reduced simple normal crossings divisor. Furthermore, the monodromy along each component of $f^*D$ of $\VV':=f_1^*\VV$ is unipotent, since $m$ may be as divisible as needed. Here $f_1=f|_{f^{-1}U}$ and the pull-back is just the non-characteristic pull-back for filtered left $\dd$-modules (see Section \ref{nonchar}). Hence $\VV'$ is also a polarizable variation of Hodge structure on $f^{-1}U$. \par
Now put $\LL'=f^*\LL$ and $A'=f^*(A)$, so that
\[(\LL')^m\simeq\oo_Y(A'+m\alpha E').\] 
Hence
\[(\LL'(-E'))^{m\alpha}\simeq(\LL')^{m(\alpha-1)}(A').\] 
Since $\LL'$ is nef, one knows that the line bundle on the right-hand side is also ample. Therefore 
\[\LL'\simeq\oo_Y(H+E'),\]
for some ample divisor $H$ on $Y$. \par
It is obvious that $\VV$ is a direct summand of $f_{1, *}\VV'$.\footnote{Here the push-forward is the Hodge module direct image of $\VV$ after side-change.} Hence so is $M$ of $\mathcal H^0f_*M'$ where $M'$ is the pure Hodge module uniquely determined by $\VV'$. Therefore, $S(M)$ is also a direct summand of $f_*(S(M'))$. Since $f$ is finite, it is sufficient to prove vanishing for $S(M')$ and $\LL'$ on $Y$. By Serre duality (see \eqref{Q(M)}), it is equivalent to prove
\[H^{i}(Y, Q(M'^*)[-n]\otimes (\LL')^{-1})=0, \ i<n,\]
where $M'^*$ is the pure Hodge module extended from the dual of $\VV'$ and  $n=\text{dim}Y$. Consequently, the proof is done by step 1.
\end{proof}
\section{Injectivity in the Normal Crossings Case} 
Let $X$ be a smooth complex projective variety, and $D=\sum^r_{i=1} D_i$ a reduced simple normal crossings divisor with irreducible components $D_i$. Assume $M$ is the polarizable pure Hodge module extending a polarizable variation of Hodge structure $\VV=(\vv, \fdot\vv, \VV_\QQ)$ defined on $U=X\setminus D$.
Then one uses the notations introduced in Section \ref{snc}.\par 
The result in this section owes a lot to the approach of J. Suh to the Kawamata-Viehweg vanishing theorem for Hodge modules in \cite{Suh}. The proof follows closely his argument and a similar argument of C. Schnell in \cite{Sch14a} with a slight refinement that leads to a more general injectivity statement. All of these arguments follow in turn the general strategy of Esnault-Viehweg \cite{EV} towards vanishing and injectivity theorems. \par
Let $\LL$ be a line bundle on $X$. Assume  \[\LL^N\simeq\mathcal{O}_X(D')\]
for some $N$ large enough, such that \[D'=\sum^r_{i=1} \alpha_i D_i,  \ \ \ 0\leq\alpha_i \ll N,\] and $D'\neq0$.
\begin{theorem} \label {thm1}
Let $E=\sum^r_{i=1}\mu_i D_i$ be an effective divisor such that ${\rm Supp}(E)\subset {\rm Supp}(D')$. Then for all $i$, the natural map induced by $E$
\[H^i(X, {\rm Gr}^F_{q(M)} {\rm DR}(\tilde{\mathcal V}(*D))\otimes \LL^{-1}(-E))\longrightarrow H^i(X, {\rm Gr}^F_{q(M)} {\rm DR}(\tilde{\mathcal V}(*D))\otimes \LL^{-1})\] 
is surjective.
\end{theorem}
\begin{proof}
Set 
\[\LL^{(i)^{-1}}:=\LL^{\otimes -i}(\lfloor\dfrac{iD'}{N}\rfloor).\] By \cite[Theorem 3.2]{EV}, one knows each $\LL^{(i)^{-1}}$ has a flat logarithmic connection and its residue along $D_i$ is just multiplication by $\{\dfrac{i\alpha_i}{N}\}$ (the fractional part of $\dfrac{i\alpha_i}{N}$).
\par Let $\pi: Y\longrightarrow X$ be the cyclic cover obtained by taking the $N$-th root of $D'$. See \cite[\S3]{EV} for the construction of cyclic covers. Set $\pi'=\pi|_{U'=\pi^{-1}(X\setminus D)}$. So $\pi'$ is $\acute{\text{e}}$tale. Since the relative de Rham complex is trivial, we get a variation of Hodge structure (the Gauss-Manin connection of $\pi'$, see \cite[\S10]{PS}),
\[\mathbb V_1:=(\pi'_{*}\oo_{U'}, \fdot\pi'_{ *}\oo_{U'}, \pi'_{*}\QQ_{U'}),\] with trivial filtration. \par 
Since $\pi'_{*}\oo_{U'}=\bigoplus^{N-1}_{i=0}\LL^{(i)^{-1}}|_{U'}$, we have a new variation of Hodge structure
\[\mathbb{V}\otimes\mathbb V_1=\left(\mathcal V\otimes\bigoplus^{N-1}_{i=0}\mathcal{L}^{(i)^{-1}}|_U, F^\bullet, \mathbb V_{\QQ}\otimes \pi'_*\QQ_{U'}\right),\] with 
\[F^p\left(\mathcal V\otimes\bigoplus^{N-1}_{i=0}\mathcal{L}^{(i)^{-1}}|_U\right)=F^p\mathcal V\otimes\bigoplus^{N-1}_{i=0}\mathcal{L}^{(i)^{-1}}|_U.\] 
Then we obtain a mixed Hodge module $j_*j^{-1}M_1$, where $M_1$ is the pure Hodge module corresponding to $\mathbb V\otimes \mathbb V_1$ by Theorem \ref{geneq}.
The filtered left $\dd$-module $\mathcal N$, 
\[\mathcal N:=\tilde{{\mathcal V}}(*D)\otimes\bigoplus^{N-1}_{i=0}\LL^{(i)^{-1}}(*D),\] underlies $j_*j^{-1}M_1$. Hence the filtered left $\dd$-module $\tilde{{\mathcal V}}(*D)\otimes \LL^{(1)^{-1}}(*D)$ is a direct summand of $\mathcal N$. We also know $\tilde {\mathcal V}\otimes \LL^{(1)^{-1}}$ is the Deligne canonical extension of $\mathcal V\otimes \LL^{(1)^{-1}}|_U$ because of the assumption $\alpha_i \ll N$.
Therefore, by filtered quasi-isomorphism \eqref{qi} we have a filtered quasi-isomorphism
\begin{equation}\label{qi1}
{\rm DR}_{(X,D)}(\tilde{\mathcal V}\otimes \LL^{(1)^{-1}}, \fdot)\simeq {\rm DR}(\tilde{\mathcal V}(*D)\otimes \LL^{(1)^{-1}}(*D), \fdot).
\end{equation}
\par By Corollary \ref{degen}, we see that the spectral sequence
\[E^{p,q}_1=H^{p+q}(X, {\rm Gr}^F_{-p}{\rm DR}(\mathcal N,\fdot))\implies H^{p+q}(X, {\rm DR}(\mathcal N))\]
degenerates at $E_1$. Hence as a direct summand, we know that the spectral sequence (since $\LL^{(1)^{-1}}=\LL^{-1}$)
\[E^{p,q}_1=H^{p+q}(X, {\rm Gr}^F_{-p}{\rm DR}_{(X,D)}(\tilde{\mathcal V}\otimes \LL^{-1}, \fdot))\implies H^{p+q}(X, {\rm DR}_{(X,D)}(\tilde{\mathcal V}\otimes \LL^{-1}, \fdot))\]
degenerates at $E_1$. From the filtration of $\mathbb V\otimes \mathbb V_1$ and the definition of the filtration on the Deligne canonical extension (equation \eqref{filtrDe}), it is clear that
\begin{equation}\label{iso}
{\rm Gr}^F_{p}{\rm DR}_{(X,D)}(\tilde{\mathcal V}\otimes \LL^{-1}, \fdot))={\rm Gr}^F_{p}{\rm DR}_{(X,D)}(\tilde{\mathcal V}, \fdot))\otimes \LL^{-1}.
\end{equation}
For any integer $a\leq 0$, all the eigenvalues of residue of $\tilde {\mathcal V}\otimes \LL^{-1}(aD_i)$ along $D_i$ (irreducible components of $E$) are strictly positive (\cite[Lemma 2.7]{EV}). Therefore, by repeatedly applying \cite[Lemma 2.10]{EV} for the components of $E$, we see that the natural map 
\begin{equation}\label{qwf}
 {\rm DR}_{(X,D)}(\tilde {\mathcal V}\otimes \LL^{-1}(-E))\longrightarrow {\rm DR}_{(X,D)}(\tilde {\mathcal V}\otimes \LL^{-1})
\end{equation}
is a quasi-isomorphism (without filtrations).
If one puts the trivial filtration on $\LL^{-1}(-E)$, we get another spectral sequence
\[E^{p,q}_1=H^{p+q}(X, {\rm Gr}^F_{-p}{\rm DR}_{(X,D)}(\tilde{\mathcal V}\otimes \LL^{-1}(-E), \fdot))\implies H^{p+q}(X, {\rm DR}_{(X,D)}(\tilde{\mathcal V}\otimes \LL^{-1}(-E), \fdot))\]
Clearly as before, 
\begin{equation}\label{iso1}
{\rm Gr}^F_{p}{\rm DR}_{(X,D)}(\tilde{\mathcal V}\otimes \LL^{-1}(-E), \fdot))={\rm Gr}^F_{p}{\rm DR}_{(X,D)}(\tilde{\mathcal V}, \fdot))\otimes \LL^{-1}(-E).
\end{equation}
Now this spectral sequence may not degenerate at $E_1$. However, one has the following diagram,
\begin{diagram}
 H^{p-q(M)}(X, \mathcal A)  &\lTo{\alpha}      &H^{p-q(M)}(X, \mathcal C) \\
\uTo_{\beta}   &                   &\uTo_{\gamma}\\
H^{p-q(M)}(X, \mathcal B) &\lTo         &H^{p-q(M)}(X, \mathcal D),
\end{diagram}
where $\mathcal A={\rm Gr}^F_{q(M)}{\rm DR}_{(X,D)}(\tilde{\mathcal V}\otimes \LL^{-1}, \fdot))$, $\mathcal B={\rm Gr}^F_{q(M)}{\rm DR}_{(X,D)}(\tilde{\mathcal V}\otimes \LL^{-1}(-E), \fdot))$, $\mathcal C={\rm DR}_{(X,D)}(\tilde {\mathcal V}\otimes \LL^{-1})$ and $\mathcal D={\rm DR}_{(X,D)}(\tilde {\mathcal V}\otimes \LL^{-1}(-E))$.\par
Observe that $\alpha$ is surjective because of $E_1$ degeneracy. Also, $\gamma$ is an isomorphism because of the quasi-isomorphism \eqref{qwf}. Hence,
\[\beta:H^{i}(X, {\rm Gr}^F_{q(M)}{\rm DR}_{(X,D)}(\tilde{\mathcal V}\otimes \LL^{-1}(-E), \fdot)))\longrightarrow H^{i}(X, {\rm Gr}^F_{q(M)}{\rm DR}_{(X,D)}(\tilde{\mathcal V}\otimes \LL^{-1}, \fdot))) \] is surjective for all $i$. Combining this with the filtered quasi-isomorphism \eqref{qi1} and the isomorphisms \eqref{iso} and \eqref{iso1}, the proof is finished.

\end{proof}
\begin{remark} \label{re1}
In the above proof, the condition $0<\alpha_i\ll N$ is needed because it ensures that  $\tilde {\mathcal V}\otimes \LL^{-1}$ is the Deligne canonical extension. Therefore, Theorem \ref{thm1} is still true if we require that 
\[\alpha_i+\lambda_{D_i}<1\] instead, where $\lambda_{D_i}$ is the maximal eigenvalue of the residue of $\tilde {\mathcal V}$ along $D_i$. 
\end{remark}
An injectivity statement for $S(M)$ follows from the above theorem as follows. By \eqref{dualiso},
\[H^b(X, {\rm Gr}^F_{q(M)}{\rm DR}(\tilde{\mathcal V}(*D), \fdot)\otimes \LL^{-1})=H^b(X, Q_X(M)\otimes \LL^{-1}).\] 
By Grothendieck-Serre Duality, 
\[H^i(X, Q_X(M)\otimes \LL^{-1})=H^{-i}(X, S_X(M^*)\otimes \LL)^*.\] Here $M^*$ is the pure Hodge module extended from the dual of the variation of Hodge structure $\mathbb V$. Therefore, the following corollary has been proved by replacing $M$ by $M^*$ in Theorem \ref{thm1}.	
\begin{corol} \label{inj}
Under the assumption of Theorem \ref{thm1},
\[H^i(X, S_X(M)\otimes \LL)\longrightarrow H^{i}(X, S_X(M)\otimes \LL(E))\] is injective for all $i$.
\end{corol}
\section{Injectivity in the general case}
In this section I prove Theorem \ref{injectivity} from the introduction, and further extensions.
From now on, divisors always mean Cartier divisors. When X is smooth, Cartier and Weil divisors will not be distinguished.
\begin{theorem} \label{semiample}
Let $X$ be a complex projective variety, and let $M$ be a pure Hodge module with strict support $X$. If $\LL$ is a semi-ample line bundle and $E$ an effective divisor with $H^{0}(X, \LL^{v}(-E))\neq 0$ for some $v>0$, then the natural map 
\[H^i(X, S_X(M)\otimes \LL)\longrightarrow H^i(X, S_X(M)\otimes \LL(E))\] 
is injective for all $i$. 
\end{theorem}
\begin{proof}
By assumption, write
\[\LL^v\simeq\mathcal{O}_X(E+C)\] 
for an effective divisor $C$. Take a log resolution of $E+C+\text{singular locus of }M$. (Since $M$ is extended from a polarizable variation of Hodge structure defined on a smooth Zariski open set $U$, the singular locus means $X\setminus U.$\footnote{This is not a good definition, because $U$ can be shrunk by subtracting any proper Zariski-closed set as needed.}):
\[f: X_1\longrightarrow X\]
Set $E_1=f^*E$, $C_1=f^*C$. Since $f$ is birational, the variation of Hodge structure also extends to a polarizable pure Hodge module strictly supported on $Y$, called $M_1$, and the singular locus of $M_1$ is just the exceptional divisor of $f$, which is a simple normal crossing divisor.
Clearly, $\LL_1:=f^*(\LL)$ is still semi-ample. Hence, $\LL_1^{u}$ is base-point free for some $u$ large enough. Pick a general divisor $D_1$ in $\abs{\LL_1^{u}}$ so that it is transversal to the exceptional divisor of $f$. Then \[\LL_1^{u+v}\simeq\mathcal{O}_X(D_1+E_1+C_1),\] and $D_1+E_1+C_1$ is a divisor with simple normal crossings support. So the assumption of Corollary \ref{inj} has been fulfilled. \par
By Corollary \ref{hdim},
\[R^if_*S(M_1)=
\begin{cases}
S(M)& i=0\\
0& i>0
\end{cases}.\] Hence by the degeneracy of the Leray spectral sequence and the projection formula,
\[H^b(X_1, S(M_1)\otimes f^*\LL)=H^b(X, S(M)\otimes \LL).\]
The proof is done by Corollary \ref{inj}.
\end{proof}
Recall that a line bundle $\LL$ on $X$ is $f$-semi-ample for a proper morphism $f:X\longrightarrow Y$ of varieties if the natural map \[ f^*f_*\LL^n\longrightarrow \LL^n\] is surjective for some $n>0$. 
As a corollary, we obtain the following torsion-freeness statement, generalizing Corollary \ref{torsionfree}. 
\begin{corol}
Let $f:X\longrightarrow Y$ be a surjective projective morphism of complex algebraic varieties, and let M be a pure Hodge module strictly supported on $X$. If $\LL$ is an $f$-semi-ample line bundle on $X$, then for $i\geq0$
\[R^if_*(S(M)\otimes \LL) \]
is torsion-free.
\begin{proof}
By a standard reduction process (see for instance \cite[proof of Theorem 1-2-3]{KMM}), one can assume $X$ and $Y$ are projective and $\LL$ is semi-ample. Suppose $\mathcal T$ is a torsion sub-sheaf of $R^if_*(S(M)\otimes \LL) $ for some $i$. Pick a sufficient ample line bundle $A$ on Y such that $\mathcal T\otimes A$ is globally generated and 
\[H^i(X, S(M)\otimes L\otimes f^*A)=H^0(Y, A\otimes R^if_*(S(M)\otimes L)).\]
Hence, by Theorem \ref{inj}, for any effective divisor $D$ on Y, the natural map
\[H^0(Y, A\otimes R^if_*(S(M)\otimes \LL))\longrightarrow H^0(Y, A(D)\otimes R^if_*(S(M)\otimes \LL))\]
is injective.
Choose $D$ so that multiplication by the equation of $D$ kills $\mathcal T$. Then $H^0(Y, A\otimes \mathcal T)=0$, which implies $\mathcal T=0$.   
\end{proof}
\end{corol}
Similarly, we also get an injectivity theorem for nef and big line bundles.
\begin{theorem}\label{nefbig}
Let X be a complex projective variety, and let $M$ be a pure Hodge module with strict support X. If $\LL$ is a nef and big line bundle and $E$ an effective divisor, then
\[H^i(X, S(M)\otimes \LL)\longrightarrow H^i(X, S(M)\otimes \LL(E))\] 
is injective for all $i$.
\end{theorem}
\begin{proof}
Since $\LL$ is nef and big, we can write
\[\LL^v\simeq\mathcal{O}_X(A+D+E)\]
for an ample divisor $A$, an effective divisor $D$ and some $v>0$. 
By taking a log resolution for $D+E+\text{singular locus of }M$, it is enough to assume that $X$ is smooth and the singular locus of $M$ is a simple normal crossings divisor containing ${\rm Supp}(D+E)$. Choosing $N$ large enough, then
\[\LL^N\simeq\LL^{N-v}(A)\otimes \mathcal{O}_X(D+E).\]
Since $\LL$ is nef, $H:=\LL^{N-v}(A)$ is ample. Hence  
\[\LL^{mN}\simeq H^m\otimes \mathcal{O}_X(mD+mE)\simeq\mathcal{O}_X(D_1+mD+mE)\] for some sufficiently general $D_1\in \abs{H^m}$, so that all the transversality conditions are satisfied. The statement of the theorem follows from Corollary \ref{inj}.
\end{proof}
\begin{remark}{(Injectivity for nef and abundant line bundles.)} After a series of blowing-ups, nef and abundant ($\kappa(\LL)=v(\LL)$) line bundles can be reduced to semi-ample ones. Therefore, Theorem \ref{nefbig} is also true if $\LL$ is only nef and abundant. (More details can be found in \cite[\S5] {EV}.) 
\end{remark}
Choosing $E$ as a multiple of a very ample divisor, by Serre vanishing we obtain another proof of Theorem \ref{KVV}. 
\begin{corol} \label{kv}
If X is a complex projective variety, $\LL$ a nef and big line bundle, and M a pure Hodge module with strict support X, then
\[H^i(X, S(M)\otimes \LL)= 0, \ \ \ i>0.\]
More generally, if $\LL$ is nef only, then
\[H^i(X, S(M)\otimes \LL)= 0, \ \ \ i>n-\kappa(\LL).\]
\end{corol} 
\begin{proof}
First, $\kappa(\LL)\leq\kappa(f^*\LL)$ for any birational morphism $f: Y\longrightarrow X .$
Hence, it is enough to assume $X$ is smooth, $S_X(M)$ is locally free and $\kappa(\LL)<n$. Choose a general hyperplane section $H$ of a very ample line bundle, such that
\[H^b(X, S(M)\otimes \LL(H))=0, \ \ \ b>0,\]
and $H$ is non-characteristic for $M$.
Therefore, we get a short exact sequence,
\[0\longrightarrow S(M)\otimes \LL\longrightarrow S_X(M)\otimes \LL(H)\longrightarrow S(M)|_H\otimes \mathcal O_H(H)\otimes \LL|_H\longrightarrow 0.\]
Hence, by passing to the long exact sequence of cohomology,
\[H^{i-1}(H, S(M)|_H\otimes \mathcal O_H(H)\otimes \LL|_H)=H^i(X, S(M)\otimes \LL).\]
Since $\kappa(\LL|_H)\geq \kappa(\LL)$ and $ S(M)|_H\otimes \mathcal O_H(H)\simeq S(i^*(M))$, both groups vanish for $i>n-\kappa(\LL)$ by induction on dimension. 
\end{proof}

Since the eigenvalues of the Deligne canonical extension lie in $[0, 1)$, the injectivity theorem is still true if the nef and big line bundle is perturbed by an effective $\mathbb Q$-divisor of sufficiently small coefficients in the following sense.
\begin{theorem}
Let X be a projective variety with a polarizable pure Hodge module $M$ strictly supported on X, and let $N$ be a nef and big $\QQ$-divisor, $D$ and $B$ two effective divisors. There exists an $\varepsilon=\varepsilon(M, N, D, B)>0$ such that if a line bundle $\LL\sim_{\QQ}N+ \varepsilon_1 D$ for $0<\varepsilon_1<\varepsilon$, then the natural map 
\[H^i(X, S(M)\otimes \LL)\longrightarrow H^i(X, S(M)\otimes \LL(B)) \]
is injective for all $i$.
\end{theorem}   
\begin{proof}
Since $N$ is nef and big, by Kodaira's lemma,
\[nN\sim A+C+B\]
for some $0<n\in \ZZ$, $A$ an ample divisor and $C$ an effective divisor. 
Take a log-resolution for $C+B+D+\text{singular locus of }M$, 
\[f: Y\longrightarrow X.\]
Then $A_1:=f^*pA-E$ is ample for some $p\ll0$ and $E$ an effective divisor with simple normal crossings support. Hence,
\[f^*npN\sim A_1+E+pf^*C+pf^*B.\]
Write $f^*D=\sum d_iD'_i$.  
Then we can assume $U$ is Zariski open in $Y$ and $Y\setminus U$ is a simple normal crossings divisor containing the support of all effective divisors that appear, and there is a variation of Hodge structure $\VV$ on $U$ which corresponds to $M$ by the equivalence of Theorem \ref{geneq}. Take $\varepsilon$ to be 
\[\varepsilon=\min\{\dfrac{1-\lambda_{D'_i}}{d_i}\},\] where $\lambda_{D'_i}$ is the maximal eigenvalue of the residue of the Deligne canonical extension of $\VV$ along $D'_i$.
If a line bundle $\LL\sim_{\QQ}N+ \varepsilon_1 D$ for $0<\varepsilon_1<\varepsilon$, then
\[f^*\LL^k\sim \dfrac{k}{np}(A_1+E+pf^*C+pf^*B)+k\varepsilon_1 f^*D,\] where $k$ is some big enough multiple of $np$. Since $f^*\LL^l(-\varepsilon_1 lD)$ is a nef line bundle for $l\ll0$ and sufficiently divisible, 
\[f^*\LL^{k+l}\sim A_2+\dfrac{k}{np}(E+pf^*C+pf^*B)+(k+l)\varepsilon_1 f^*D,\] for some other ample divisor $A_2$. Hence
\[f^*\LL^{(k+l)m}\sim H+\dfrac{km}{np}(E+pf^*C+pf^*B)+(k+l)m\varepsilon_1 f^*D,\] where $H\in |\oo(mA_2)|$ is sufficiently general.
The statement then follows from Corollary \ref{inj} and Remark \ref{re1}.

\end{proof}
Note that $\varepsilon$ in the statement depends on the choice of a resolution, but can be made effective once one has been fixed; see also Remark \ref{qdiv} below.
A similar argument works for the following Kawamata-Viehweg type vanishing for $\QQ$-divisors.
\begin{theorem}
Let $X$ be a projective variety with a pure Hodge module $M$ strictly supported on X, and let $N$ be a nef and big $\QQ$-divisor and $D$ an effective divisor. There exists an $\varepsilon=\varepsilon(M, N, D)>0$ such that if a line bundle $\LL\sim_{\QQ}N+\varepsilon_1 D$ for $0<\varepsilon_1<\varepsilon$, then 
\[H^i(X, S(M)\otimes \LL)=0,\ \ \ i>0.\]
\end{theorem}
\begin{remark}\label{qdiv}
If X is a smooth projective variety, $M=\QQ^H_X$ (or more generally for any smooth $M$, i.e. $M$ corresponding to a polarizable variation of Hodge structure defined on $X$), and $D$ is a reduced simple normal crossings divisor, then since there is no residue under these assumptions, $\varepsilon=1$ by \cite[Theorem 9.4.17(i)]{Laz}. Hence the above theorem reduces to the original Kawamata-Viehweg vanishing for $\QQ$-divisors \cite[Theorem 9.1.18]{Laz}, or a special case of \cite[Theorem 11.1]{Pop}.   
\end{remark}

\bibliographystyle{amsalpha}

\end{document}